\documentclass[reqno, 11pt, a4paper]{amsart}

\usepackage{latexsym,ifthen,xspace}
\usepackage{amsmath,amssymb,amsthm,bbm}
\usepackage{enumerate, calc}
\usepackage[colorlinks=true, pdfstartview=FitV, linkcolor=blue,
  citecolor=blue, urlcolor=blue]{hyperref}
\usepackage[text={33pc,605pt},centering]{geometry}    
\usepackage{graphicx, xcolor}

\newtheorem{theorem}{Theorem}[section]
\newtheorem{lemma}[theorem]{Lemma}
\newtheorem{prop}[theorem]{Proposition}
\newtheorem{assumption}[theorem]{Assumption}
\newtheorem{coro}[theorem]{Corollary}

\theoremstyle{definition}

\theoremstyle{remark}
\newtheorem{remark}[theorem]{Remark}

\numberwithin{equation}{section}

\DeclareMathAlphabet{\mathsl}{OT1}{cmss}{m}{sl}
\SetMathAlphabet{\mathsl}{bold}{OT1}{cmss}{bx}{sl}

\newcommand{\al}{\ensuremath{\alpha}}

\newcommand{\ga}{\ensuremath{\gamma}}
\newcommand{\de}{\ensuremath{\delta}}

\newcommand{\ze}{\ensuremath{\zeta}}

\renewcommand{\th}{\ensuremath{\theta}}

\newcommand{\ka}{\ensuremath{\kappa}}
\newcommand{\la}{\ensuremath{\lambda}}

\newcommand{\si}{\ensuremath{\sigma}}

\newcommand{\om}{\ensuremath{\omega}}

\newcommand{\ve}{\ensuremath{\varepsilon}}

\newcommand{\Ga}{\ensuremath{\Gamma}}
\newcommand{\De}{\ensuremath{\Delta}}

\newcommand{\Om}{\ensuremath{\Omega}}
\newcommand{\cB}{\ensuremath{\mathcal B}}

\newcommand{\cE}{\ensuremath{\mathcal E}}
\newcommand{\cF}{\ensuremath{\mathcal F}}

\newcommand{\cL}{\ensuremath{\mathcal L}}

\newcommand{\bbN}{\ensuremath{\mathbb N}}

\newcommand{\bbR}{\ensuremath{\mathbb R}}

\newcommand{\bbZ}{\ensuremath{\mathbb Z}} 
\newcommand{\me}{\ensuremath{\mathrm{e}}}

\newcommand{\md}{\ensuremath{\mathrm{d}}}

\newcommand{\scpr}[3]{%
  \ensuremath{%
    \big\langle
      #1, #2
    \big\rangle_{\raisebox{0.1ex}{$\scriptstyle \ell^{\raisebox{.1ex}{$\scriptscriptstyle 2$}} (#3)$}}
  }
}
\newcommand{\norm}[3]{%
  \ensuremath{%
    \big\lVert
      #1
    \big\rVert_{\raisebox{-.0ex}{$\scriptstyle \ell^{\raisebox{.2ex}{$\scriptscriptstyle #2$}} (#3)$}}
  }
}
\newcommand{\Norm}[2]{%
  \ensuremath{%
    \big\lVert
      #1
    \big\rVert_{\raisebox{-.0ex}{$\scriptstyle #2$}}
  }
}

\DeclareMathOperator{\mean}{\mathbb{E}}

\DeclareMathOperator{\prob}{\mathbb{P}}
\DeclareMathOperator{\Prob}{\mathrm{P}}

\DeclareMathOperator{\supp}{\mathrm{supp}}

\DeclareMathOperator{\argmax}{\mathrm{arg max}}

\newcommand{\av}[1]{\mathop{\mathrm{av}}(#1)}

\newcommand{\ldef}{\ensuremath{\mathrel{\mathop:}=}}
\newcommand{\rdef}{\ensuremath{=\mathrel{\mathop:}}}

\newcommand{\indicator}{\ensuremath{\mathbbm{1}}}

\begin{document}

\title[Heat kernel and intrinsic metric for random walks with general speed measure]{Heat kernel estimates and intrinsic metric for random walks with general speed measure under degenerate conductances}


\author{Sebastian Andres}
\address{University of Cambridge}
\curraddr{Centre for Mathematical Sciences, Wilberforce Road, Cambridge CB3 0WB}
\email{s.andres@statslab.cam.ac.uk}
\thanks{}

\author{Jean-Dominique Deuschel}
\address{Technische Universit\"at Berlin}
\curraddr{Strasse des 17. Juni 136, 10623 Berlin}
\email{deuschel@math.tu-berlin.de}
\thanks{}

\author{Martin Slowik}
\address{Technische Universit\"at Berlin}
\curraddr{Strasse des 17. Juni 136, 10623 Berlin}
\email{slowik@math.tu-berlin.de}
\thanks{}



\date{\today}

\dedicatory{}

\begin{abstract}
  We establish heat kernel upper bounds for a continuous-time random walk under unbounded conductances satisfying an integrability assumption, where we correct and extend recent results in \cite{ADS16a} to a general class of speed measures. The resulting heat kernel estimates are governed by the intrinsic metric induced by the speed measure. We also provide a comparison result of this metric with the usual graph distance, which is optimal in the context of the random conductance model with ergodic conductances.  
\end{abstract}

\maketitle

\tableofcontents

\section{Introduction}
Let $G = (V, E)$ be an infinite, connected, locally finite graph with vertex set $V$ and (non-oriented) edge set $E$.  We will write $x \sim y$ if $\{x,y\} \in E$.  Consider a family of positive weights $\om = \{\om(e) \in (0, \infty) : e \in E\}\in \Omega$, where $\Omega=\bbR_+^E$ is the set of all possible configurations. We also refer to $\om(e)$ as the \emph{conductance} of the edge $e$.  With an abuse of notation, for $x, y \in V$ we set $\om(x,y) = \om(y,x) = \om(\{x,y\})$ if $\{x,y\} \in E$ and $\om(x,y) = 0$ otherwise. 
Let us further define measures $\mu^{\om}$ and $\nu^{\om}$ on $V$ by
\begin{align*}
  \mu^{\om}(x) \;\ldef\; \sum_{y \sim x}\, \om(x,y)
  \qquad \text{and} \qquad
  \nu^{\om}(x) \;\ldef\; \sum_{y \sim x}\, \frac{1}{\om(x,y)}.
\end{align*}
Given a speed measure $\th\!: V \to (0, \infty)$ we consider a continuous time  continuous time Markov chain, $X = \{X_t\!: t \geq 0 \}$, on $V$ with generator $\cL_{\th}^{\om}$ acting on bounded functions $f\!:  V \to \bbR$ as
\begin{align} \label{def:L}
  \big(\cL_{\th}^{\om} f)(x)
  \;=\; 
  \frac 1 {\th(x)} \, \sum_{y \sim x} \om(x,y) \, \big(f(y) - f(x)\big).
\end{align}
Then the Markov chain, $X$, is \emph{reversible} with respect to the speed measure $\th$, and regardless of the particular choice of $\th$ the jump probabilities of $X$ are given by $p^{\om}(x,y) \ldef \om(x,y) / \mu^{\om}(x)$, $x,y \in V$, and the various random walks corresponding to different speed measures will be time-changes of each other.  The maybe most natural choice for the speed measure is $\th \equiv \th^\om = \mu^\om$, for which we obtain the \emph{constant speed random walk} (CSRW) that spends i.i.d.\ $\mathop{\mathrm{Exp}}(1)$-distributed waiting times at all visited vertices.  Another frequently arising choice for $\th$ is the counting measure, i.e.\ $\th(x) = 1$ for all $x \in V$, under which the random walk waits at $x$ an exponential time with mean $1/\mu^\om(x)$.  Since the law of the waiting times does  depend on the location, $X$ is also called the \emph{variable speed random walk} (VSRW).  

For any choice of $\th$ we denote by $\Prob_x^{\om}$ the law of the process $X$ starting at the vertex $x \in V$.  For $x, y \in V$ and $t \geq 0$ let $p_{\th}^{\om}(t,x,y)$ be the transition densities of $X$ with respect to the reversible measure (or the \emph{heat kernel} associated with $\cL_\th^{\om}$), i.e.\
\begin{align*}
  p_{\th}^{\om}(t,x,y)
  \;\ldef\;
 \frac{\Prob_x^{\om}\big[X_t = y\big]}{\th(y)}.
\end{align*}
As one of our main results we establish upper bounds on the heat kernel under a certain integrability condition on the conductances, see Theorem~\ref{thm:hke} below. The resulting bounds are of Gaussian type apart from an additional factor which may vanish for specific choices of the speed measure or the conductances (see Remark~\ref{rem:hke} below). It is well known that Gaussian bounds hold, for instance, for the CSRW on  infinite weighted graphs with bounded vertex degree in the uniformly elliptic case, that is $c^{-1} \leq \om(e) \leq c$ for all $e\in E$ for some $c\geq 1$, see \cite{De99}. More recently, Folz showed in \cite{Fo11} upper Gaussian estimates for elliptic random walk for general speed measures that need to be bounded away from zero, provided on-diagonal upper bounds at two vertices are given. In \cite{ADS16a} we weakened the uniform ellipticity condition and showed heat kernel upper bounds for the CSRW and VSRW under a similar integrability condition as in Theorem~\ref{thm:hke}, while in the present paper we extend this result to general speed measures. Notice that some integrability assumption on the conductances is necessary for Gaussian bounds to hold. In fact, it is well known that  due to a trapping phenomenon  under random i.i.d.\ conductances with sufficiently heavy tails at the zero the subdiffusive heat kernel decay may occur, see \cite{BBHK08, BO12} and  cf.\ \cite{BKM15}.  For the proof of Theorem~\ref{thm:hke} we use the same strategy as in \cite{ADS16a} which is based on a combination of Davies' perturbation method (cf.\ e.g.\ \cite{Da89, Da93, CKS87}) with a Moser iteration following an idea in \cite{Zh11}. We refer to \cite[Section~1.2]{ADS16a} for a more detailed outline of the method.

Naturally, the heat kernel upper bounds  in Theorem~\ref{thm:hke} are  governed by the distance function  $d_{\th}^{\om}$ on $V \times V$ defined by
\begin{align} \label{eq:def_chemdist}
  d_{\th}^{\om}(x,y)
  \;\ldef\;
  \inf_{\gamma\in \Gamma_{xy}}
  \Bigg\{
    \sum_{i=0}^{l_{\gamma}-1} 
    \bigg( 
      1 \wedge \frac{\th(z_i) \wedge \th(z_{i+1})}{\om(z_i,z_{i+1})}
    \bigg)^{\!\!1/2}
  \Bigg\},
\end{align}
where $\Ga_{xy}$ is the set of all nearest-neighbor paths $\ga = (z_0, \ldots, z_{l_\ga})$  connecting $x$ and $y$ (cf.\ \cite{Da93, BD10, Fo11, Mo09, ADS16a}).  Note that $d_{\th}^{\om}$ is a metric which is adapted to the transition rates and the speed measure of the random walk. Further, for the CSRW, i.e.\ $\th \equiv \th^{\om} = \mu^{\om}$, the metric $d_{\th}^{\om}$ coincides with the usual graph distance $d$.  In general, $d_{\th}^{\om}$ can be identified with the intrinsic metric generated by the Dirichlet form associated with $\cL_{\th}^{\om}$ and $X$, see Proposition~\ref{prop:intr_metric} below.  Further, notice that $d_\theta^\om(x,y) \leq d(x,y)$ for all $x,y \in V$. In fact, the distance $d_\theta^\om$ can become much smaller than the graph distance, see \cite[Lemma~1.12]{ADS16a} for an example in the context of a VSRW under random conductances. As our second main result, stated in Theorem~\ref{thm:dist} below, for any $x,y\in V$ sufficiently far apart, we provide under a suitable integrability condition on $\om$ a lower bound on $d_\theta^\om(x,y)$ in terms of a certain power of $d(x,y)$. This lower bound turns out to be optimal within our general framework up to an arbitrarily small correction in the exponent.  

The rest of the paper is organised as follows.   In Section~\ref{sec:comp_dist} we prove a lower bound on $d_\theta^\om$ in terms of the graph distance and we discuss its optimality by providing an example in the context of the random conductance model on $\bbZ^d$.  In Section~\ref{sec:UHK} we show the heat kernel upper bounds. Throughout the paper we write $c$ to denote a positive constant which may change on each appearance. Constants denoted $C_i$ will be the same through each argument.

\section{Comparison result for the intrinsic metric and its optimality}
\label{sec:comp_dist}

\subsection{Preliminaries} \label{sec:prelim}
The graph $G$ is endowed with the counting measure, i.e.\ the measure of $A \subset V$ is simply the number $|A|$ of elements in $A$.  Further, we denote by $B(x,r)$ the closed ball with center $x$ and radius $r$ with respect to the natural graph distance $d$, that is $B(x,r) \ldef \{y \in V \mid d(x,y) \leq r\}$. Throughout the paper we will make the following assumption on $G$.
\begin{assumption}\label{ass:graph}
  The graph $G$ satisfies the following conditions.
  \begin{itemize}
  \item[(i)] Uniformly bounded vertex degree, that is there exists $C_{\mathrm{deg}} \in [1, \infty)$ such that
    \begin{align} \label{eq:ass:degree}
      |\{ y: y \sim x \}| 
      \;\leq\; 
      C_{\mathrm{deg}}, 
      \qquad \forall x \in V.
    \end{align}
  \item[(ii)] Volume regularity of order $d$ for large balls, that is there exist $d \geq 2$ and $C_{\mathrm{reg}} \in (0, \infty)$ such that for all $x \in V$ there exists $N_1(x) < \infty$ with
    \begin{align}\label{eq:ass:vd}
      C^{-1}_{\mathrm{reg}}\, n^d  \;\leq\; |B(x,n)| \;\leq\; C_{\mathrm{reg}}\, n^d,
      \qquad \forall\, n \geq N_1(x). 
    \end{align}
  \item[(iii)] Local Sobolev inequality $(S_{d'}^1)$ for large balls, that is there exists $d' \geq d$ and $C_{\mathrm{S_1}} \in (0, \infty)$ such that for all $x \in V$ the following holds. There exists $N_2(x) < \infty$ such that for all  $n \geq N_2(x)$, 
    \begin{align}\label{eq:ass:riso}
      \Bigg(\sum_{y \in B(x,n)}\! |u(y)|^{\frac{d'}{d'-1}}\Bigg)^{\!\!\frac{d'-1}{d'}}
      \;\leq\;
      C_{\mathrm{S_1}}\, n^{1 - \frac{d}{d'}}\mspace{-6mu}
      \sum_{\substack{y \vee z \in B(x,n)\\ \{y,z\} \in E}}\mspace{-6mu}
      \big|u(y) - u(z) \big|
    \end{align}
    for all $u\! : V \to \bbR$ with $\supp u \subset B(x,n)$.
  \end{itemize}
\end{assumption}
\begin{remark}\label{rem:graph:Zd}
  The Euclidean lattice, $(\bbZ^d, E_d)$, satisfies the Assumption~\ref{ass:graph} with $d' = d$ and $N_1(x) = N_2(x) = 1$. Further, if Assumption~\ref{ass:graph} holds with $N_1(x) = N_2(x) = 1$ for all $x\in V$, then Gaussian bounds hold on the unweighted graph.
\end{remark}
For $f\!: V \to \bbR$ we define the operator $\nabla$ by
\begin{align*}
  \nabla f\!: E \to \bbR,
  \qquad
  E \ni e \;\longmapsto\; \nabla f(e) \;\ldef\; f(e^+) - f(e^-),
\end{align*}
where for each non-oriented edge $e \in E$ we specify one of its two endpoints as its initial vertex $e^+$ and the other one as its terminal vertex $e^-$.  Further, the corresponding adjoint operator $\nabla^*F\!: V \to \bbR$ acting on functions $F\!: E \to \bbR$ is defined in such a way that $\langle \nabla f, F \rangle_{\ell^{^2}\!(E)} = \langle f, \nabla^* F \rangle_{\ell^{^2}\!(V)}$ for all $f \in \ell^2(V)$ and $F \in \ell^2(E)$.  Notice that in the discrete setting the product rule reads
\begin{align}\label{eq:rule:prod}
  \nabla(f g)
  \;=\;
  \av{f} \nabla g + \av{g} \nabla f,
\end{align}
where $\av{f}(e) \ldef \frac{1}{2} (f(e^+) + f(e^-))$.  On the weighted Hilbert space $\ell^2(V, \th)$ the \emph{Dirichlet form} associated with $\cL^{\om}_{\th}$ is given by
\begin{align} \label{eq:def:dform}
  \cE^{\om}(f,g)
  \;\ldef\;
  \scpr{f}{-\cL^{\om} g}{V, \theta}
  \;=\;
  \scpr{\nabla f}{\om \nabla g}{E}
  \;=\;
  \scpr{1}{\md \Ga^{\om}(f,g)}{E},
\end{align}
where $\md \Ga^{\om}(f,g) \ldef \om \nabla f \nabla g$ and $\cE^{\om}(f) = \cE^{\om}(f,f)$. 
\smallskip

As a first step, we identify the metric $d^\om_{\th}$ as the intrinsic metric of the Dirichlet form $ \cE^{\om}$ on $\ell^2(V, \th)$. 
\begin{prop} \label{prop:intr_metric}
  For every $x, y \in V$,
  \begin{align*}
    d_{\th}^{\om}(x,y)
    \;=\; 
    \sup 
    \Big\{ 
      \psi(y) - \psi(x) : 
      \|\nabla \psi \|_{\infty} \leq 1,\,
      \md \Ga^{\om}(\psi, \psi)(e) \leq \th(e^+) \wedge \th(e^-),
      e \in E  
    \Big\}.
  \end{align*}
\end{prop}
\begin{proof}
  We follow the argument in \cite[Proposition~10.4]{Mo09}. For any $x, y \in V$ set
  \begin{align*}
    \De_{\th}^{\om}(x,y)
    \ldef
    \sup 
    \Big\{ 
      \psi(y) - \psi(x) : 
      \|\nabla \psi \|_{\infty} \leq 1,\,
      \md \Ga^{\om}(\psi, \psi)(e) \leq \th(e^+) \wedge \th(e^-),
      e \in E  
    \Big\}.
  \end{align*}
  Then, for any function $\psi\!: V \to \bbR$ with the properties that $\|\nabla \psi \|_{\infty} \leq 1$ and $\md \Ga^{\om}(\psi,\psi)(e) \leq \th(e^+) \wedge \th(e^-)$ for all $e \in E$ we obtain
  \begin{align*}
    \nabla \psi(e) 
    \;\leq\;  
    \bigg( 
      1 \wedge \frac{\th(e^+) \wedge \th(e^-)}{\om(e)} 
    \bigg)^{\!\!1/2}.
  \end{align*}
  Let $\ga \in \Ga_{x,y}$ be a nearest neighbour path connecting $x$ and $y$.  By summing over all consecutive vertices in $\ga$, we get that $\psi(y) - \psi(x) = \sum_{i=0}^{l_{\ga}-1} \psi(z_{i+1}) - \psi(z_i)$. Thus,
  \begin{align*}
    \De_{\th}^{\om}(x,y)
    \;\leq\;
    d_{\th}^{\om}(x,y).
  \end{align*}
  In order to obtain $\De_{\th}^{\om}(x,y) \geq d_{\th}^{\om}(x,y)$, set $\psi(z) \ldef d_{\th}^{\om}(x,z)$ for all $z \in V$. Then, for any edge $e \in E$ an application of the triangle inequality and the definition of $d_{\th}^{\om}$ yields 
  \begin{align*}
    \big| \nabla \psi(e) \big| 
    \;\leq\; 
    \big| d_{\th}^{\om}(x, e^+) - d_{\th}^{\om}(x, e^-) \big| 
    \;\leq\;
    d_{\th}^{\om}(e^+, e^-) 
    \;\leq\; 
    1.
  \end{align*}
  Likewise, it follows that, for any $e \in E$,
  \begin{align*}
    \md \Ga^{\om}(\psi, \psi)(e)
    \;\leq\;
    \om(e)\, d_{\th}^{\om}(e^+, e^-)^2
    \;\leq\;
    \th(e^+) \wedge \th(e^-).
  \end{align*}
  Thus, $\psi$ satisfies the requirements in the definition of $\De_{\th}^{\om}(x,y)$. Since $\psi(x) = 0$ we finally have that $d_{\th}^{\om}(x,y) \leq \De_{\th}^{\om}(x,y)$.
\end{proof}
For some $\phi\!: V \to [0, \infty)$, $p \in [1, \infty)$ and any non-empty, finite $B \subset V$, we define space-averaged weighted $\ell^{p}$-norms on functions $f\!: B \to \bbR$ by 
\begin{align*}
  \Norm{f}{p, B, \phi}
  \;\ldef\;
  \bigg(
    \frac{1}{|B|}\; \sum_{x \in B}\, |f(x)|^p\, \phi(x)
  \bigg)^{\!\!1/p}
  \qquad \text{and} \qquad
  \Norm{f}{\infty, B} \;\ldef\; \max_{x \in B} |f(x)|.
\end{align*}
If $\phi \equiv 1$, we simply write $\|f\|_{p, B} \ldef \|f\|_{p, B, \phi}$.

\subsection{Lower bound on $d_\theta^\om$}
As our first main result we show that on a large scale the metric $d_{\th}^{\om}$ can be bounded from below by a certain power of the graph distance $d$.
\begin{theorem} \label{thm:dist}
  Let $p > (d-1)/2$ and assume that for any $x\in V$,
  \begin{align} \label{eq:ass_dist}
    m_{p} 
    \;\ldef\; 
    \limsup_{n \to \infty} \Norm{1 \vee \mu^{\om} / \th}{p, B(x, n)}
    \;<\;
    \infty.
  \end{align}
  Then, there exists $c(m_{p}) > 0$ such that the following holds. For every $x \in V$ there exists $N_3(\om, x) < \infty$ such that  for any $y \in V$ with $d(x,y) \geq N_3(\om, x)$,
  \begin{align}
    d_{\th}^{\om}(x,y)
    \;\geq\;
    c(m_p)\, d(x,y)^{1-\frac{d-1}{2p}}.
  \end{align}
\end{theorem}
\begin{proof}
  In order to simplify notation, set $m_{\th}^{\om}(x) \ldef 1 \vee \mu^{\om}(x) / \th(x)$ for $x \in V$.  Since the function $t \mapsto 1/\sqrt{t}$ is convex, an application of the Jensen inequality yields
  \begin{align*}
    d_{\th}^{\om}(x,y)
    \;\geq\;
    \inf_{\ga \in \Ga_{x,y}} 
    l_{\ga}\,
    \Bigg(
      \frac{1}{l_{\ga}} 
      \sum_{i=0}^{l_{\ga}-1} m_{\th}^{\om}(z_i) \vee m_{\th}^{\om}(z_{i+1}) 
    \Bigg)^{\!\!-1/2}\mspace{-28mu}.
  \end{align*}
  Moreover, for any $p > (d-1)/2$, an application of H\"older's inequality yields 
  \begin{align*}
    \frac{1}{l_{\ga}} 
    \sum_{i=0}^{l_{\ga}-1} m_{\th}^{\om}(z_i) \vee m_{\th}^{\om}(z_{i+1})
    &\;\leq\;
    \frac{2 |B(x, l_{\ga})|}{l_{\ga}}\; 
    \Norm{\indicator_{\ga}\, m_{\th}^{\om}}{1, B(x, l_{\ga})}
    \\[.5ex]
    &\;\leq\;
    2\, \bigg(\frac{|B(x, l_{\ga})|}{|l_{\ga}|}\bigg)^{\!\!1/p}\,
    \Norm{m_{\th}^{\om}}{p, B(x, l_{\ga})}.
  \end{align*}
  By combining the estimates and using \eqref{eq:ass:vd} 
  there exists $c(m_{p}) > 0$ and $N_3(\om, x) < \infty$ such that for any $y \in V$ with $d(x,y) \geq N_3(\om, x)$,
  \begin{align*}
    d_{\th}^{\om}(x,y)
    \;\geq\;
    c(m_{p})\,
    \inf_{\ga \in \Ga_{x,y}} (l_{\ga})^{1-\frac{d-1}{2p}}
    \;=\;
    c(m_p)\, d(x,y)^{1-\frac{d-1}{2p}},
  \end{align*}
  where we used in the last step that $p > (d-1)/2$.
\end{proof}

\subsection{Optimality of the bound in Theorem~\ref{thm:dist}} \label{sec:example}
In this subsection we provide an example for which the lower bound in Theorem~\ref{thm:dist} is attained up to an arbitrarily small correction in the exponent.  For this purpose, consider the $d$-dimensional Euclidean lattice $(\bbZ^d, E_d)$ with $d \geq 2$, where $E_d$ denotes the set of all non-oriented nearest neighbour bonds.  As pointed out in Remark~\ref{rem:graph:Zd}, $(\bbZ^d, E_d)$ satisfies the Assumption \ref{ass:graph}.  Further, let $\prob$ be a probability measure on the measurable space $(\Om, \cF) = \big(\bbR_+^{E_d}, \cB(\bbR_+)^{\otimes\, E_d}\big)$ and write $\mean$ for the expectation with respect to $\prob$.  The space shift by $z \in \bbZ^d$ is the map $\tau_z\!: \Om \to \Om$ defined by $ (\tau_z \om)(\{x, y\}) \ldef \om(\{x+z, y+z\})$ for all  $\{x,y\} \in E_d$.  Now assume that $\prob$ satisfies the following conditions.
\begin{enumerate}[(i)]
\item $\prob$ is ergodic with respect to translations of $\bbZ^d$, i.e.\ $\prob \circ\, \tau_x^{-1} \!= \prob\,$ for all $x \in \bbZ^d$ and $\prob[A] \in \{0,1\}\,$ for any $A \in \cF$ such that $\tau_x(A) = A\,$ for all $x \in \bbZ^d$.
\item There exist $p>(d-1)/2$ such that $\mean[\om(e)^p] < \infty$ for any $e \in E_d$.
\end{enumerate}
Then, the spatial ergodic theorem gives that for $\prob$-a.e.\ $\om$,
\begin{align*}
  \lim_{n \to \infty} \Norm{\mu^{\om}}{p, B(n)}^p
  \;=\;
  \mean\!\big[ \mu^{\om}(0)^p\big]
  \;<\;
  \infty.
\end{align*}
In particular, by choosing $\th \equiv 1$, the assumption \eqref{eq:ass_dist} in Theorem~\ref{thm:dist} is fulfilled for $\prob$-a.e.\ $\om$ and the lower bound on $d_\theta^\om$ holds. Nevertheless, for general ergodic environments we cannot control the size of the random variable $N_3(x)$, $x \in \bbZ^d$, as this requires some information on the speed of convergence in the ergodic theorem. However, if we additionally assume, for instance, that the environment satisfies a  concentration inequality in form of a spectral gap inequality w.r.t.\ the so-called vertical derivative, then $\mean[N_3(x)^n]<\infty$ provided a stronger moment condition holds (depending on $n$), see Assumption~1.3 and Lemma~2.10 in \cite{AN17}.
\begin{theorem} \label{thm:example_dist}
  Consider the VSRW, i.e.\ $\th \equiv 1$. For any $p > 1$ there exists an environment of ergodic random conductances $\{\om(e) : e \in E_d\}$ on $(\bbZ^d, E_d)$ satisfying $\mean[\om(e)^p] < \infty$ such that for any $\al > p$ and $\prob$-a.e.\ $\om$ the following hold. 
  \begin{enumerate}[(i)]
  \item Suppose $d = 2$.  There exists $L_0 = L_0(\om) < \infty$ such that for all $L \geq L_0$ there exists $x = x(\om) \in \bbZ^d$ with $d(0,x) = L$ and
    \begin{align*}
      d_{\th}^{\om}(0, x)
      \;\leq\;
      c\, d(0, x)^{1 - \frac{d-1}{2\al}}.
    \end{align*}
  \item Suppose $d \geq 2$.  There exists $L_0 = L_0(\om) < \infty$ such that for all $L \geq L_0$ there exist $x = x(\om), y = y(\om) \in [-2L,2L]^d$ with $d(x,y) = L$ and
    \begin{align*}
      d_{\th}^{\om}(x, y)
      \;\leq\;
      c\, d(x, y)^{1-\frac{d-1}{2\al}}.
    \end{align*}
  \end{enumerate}
\end{theorem}
\begin{proof}
  Let $\{ Y(i, y) : i \in \{1, \ldots, d\}, y \in \bbZ^{d-1} \}$ be a family of non-negative, independent and identically distributed random variables such that
  \begin{align*}
    \prob\!\big[Y(i,y) > r\big] \;=\; r^{-\al_0 + o(1)}
    \qquad \text{as $r \to \infty$}
  \end{align*}
  for some $\al_0 > p$.  For any $x \in \bbZ^d$ we write $\hat{x}^i$ to denote the element of $ \bbZ^{d-1}$ obtained by removing the $i$-th component from $x$. Further, set
  \begin{align*}
    \om(\{x, x \pm \boldsymbol{e}_i\})
    \;\ldef\;
    Y(i, \hat{x}^{i}),
    \qquad \forall\, x \in \bbZ^d,\; i \in \{1, \ldots, d\},
  \end{align*}
  where $\{\boldsymbol{e}_1, \ldots, \boldsymbol{e}_d \}$ denotes the canonical basis in $\bbR^d$.  Then, note that the conductances are constant along the lines, but independent between different lines.  W.l.o.g.\ we further assume that $\om(e) \geq 1$ for any $e \in E_d$. We refer to \cite[Example~1.11]{ADS16a} and \cite[Section~2.2]{DF16} for a similar but different example for a model with layered random conductances.

  \textit{(i)} Consider the nearest-neighbour path $(x_n : n \geq 0)$ on $\bbN_0^d$ defined by $x_0 \ldef 0$, $x_{n+1} \ldef x_n + \boldsymbol{e}_{i_{n+1}}$ with
  \begin{align*}
    i_1 
    \;\ldef\; 
    \underset{i = 1,\ldots,d-1}{\argmax}\, \om(\{0, \boldsymbol{e}_i\}),
    \qquad
    i_{n+1}
    \;\ldef\; 
    \underset{i = 1, \ldots, d}{\argmax}\; \om(\{x_n, x_n + \boldsymbol{e}_i\}),
    \quad n \geq 1.
  \end{align*}
  In view of the definition of $d_\theta^\om$ in \eqref{eq:def_chemdist} it suffices to show that for any $\al > \al_0$ there exists $L_0 = L_0(\om) < \infty$ such that
  \begin{align} \label{eq:chem_dist_simple}
    \sum_{n=0}^L\, \om(\{x_n, x_{n+1}\})^{-1/2} 
    \;\leq\; 
    c\,
    L^{1-\frac{d-1}{2\al}}, \qquad \forall\, L \geq L_0.
  \end{align}
  For that purpose, set $M_n \ldef \max_{1 \leq k \leq n} \om(\{x_{k-1}, x_{k}\})$ and $u_n \ldef n^{1/\al}$ for  $n \geq 1$.  Then, by construction $M_n$ is the maximum of $n$ i.i.d.\ random variables $Z_1, \ldots, Z_n$ defined by 
  \begin{align*}
    Z_1
    \;\ldef\; 
    \max_{i \in \{1, \ldots,d-1\}} \om(\{0, \boldsymbol{e}_i\}), 
    \qquad 
    Z_k
    \;\ldef\; 
    \max_{i \in \{1,\ldots,d\} \setminus \{i_{k-1}\}} 
    \om(\{x_{k}, x_{k} + \boldsymbol{e}_i\}), 
    \quad k \geq 2.
  \end{align*} 
  An elementary computation shows that $\prob[Z_1 >u_k] \leq (d-1) \prob[\om(e) > u_k] \to 0$, $k \prob[Z_1 > u_k] \geq k \prob[\om(e) > u_k] \to \infty$ as $k \to \infty$ and
  \begin{align*}
    \sum_{k=1}^{\infty} 
    \prob\big[Z_1 > u_k\big]\, \exp\!\Big(\!-\! k \prob\big[Z_1 > u_k\big]\Big) 
    \;\leq\; 
    c \sum_{k=1}^\infty
    k^{-\frac{\al_0}{\al}}\, \exp\!\Big(\!-\!c\, k^{1-\frac{\al_0}{\al}} \Big) 
    \;<\; 
    \infty.
  \end{align*} 
  Thus, by \cite[Theorem~3.5.2]{EKM97} for $\prob$-a.e.\ $\om$ there exists $N_0 = N_0(\om) < \infty$ such that 
  \begin{align} \label{eq:lb_max}
    M_n \;\geq\; n^{1/\alpha}, 
    \qquad \forall\, n \geq N_0.
  \end{align}
  Let $(l_k : k \geq 0)$ be the sequence of record times defined by 
  \begin{align*}
    l_0
    \;\ldef\; 
    0 
    \qquad \text{and} \qquad 
    l_{k+1} 
    \;\ldef\; 
    \min \big\{j > l_k: M_j > M_{l_k} \big\}.
  \end{align*}
  and denote by $N(L)$ the number of records in the interval $\{0, \ldots L\}$.  Recall that
  \begin{align} \label{eq:record_limit}
    \lim_{k \to \infty} \frac{\ln l_k}{k} \;=\; 1, 
    \qquad
    \lim_{L\to \infty} \frac{N(L)}{\ln L} \; = \; 1, \qquad \text{ $\prob$-a.s.}
  \end{align}
  (cf.\ e.g.\ \cite[Section~5.4]{EKM97}).  Set $\hat{M}_k \ldef M_{l_k}$.  Using Abel's summation formula the left-hand side in \eqref{eq:chem_dist_simple} can be rewritten as
  \begin{align*}
    \sum_{n=0}^L \om(\{x_{n}, x_{n+1}\})^{-1/2} 
    &\;\leq\; 
    l_{N(L)}\, \hat{M}_{N(L)}^{-1/2} 
    \,+\, 
    \sum_{k=1}^{N(L)-1}\! l_k\, \big(\hat{M}_k^{-1/2} - \hat{M}_{k+1}^{-1/2}\big).     
  \end{align*}
  By \eqref{eq:lb_max} and \eqref{eq:record_limit} the first term is of order $L^{1-\frac{1}{2\al}}$. Further, we have that $l_k \hat{M}_k^{-1/2} \leq (l_k)^{1-\frac{1}{2\al}} \leq c\, \me^{({1-\frac{1}{2\al}})k}$ for sufficiently large $k$ and therefore
  \begin{align*}
    \sum_{k=1}^{N(L)-1}\! l_k\, \big(\hat{M}_k^{-1/2} - \hat{M}_{k+1}^{-1/2}\big)
    \;\leq\;
    \sum_{k=1}^{N(L)} l_k\, \hat M_k^{-1/2}
    \;\leq\; 
    c\, \me^{( 1 - \frac{1}{2\al}) N(L)} 
    \;\leq\; 
    c\, L^{1 - \frac{1}{2\al}}
  \end{align*}
  for all $L$ larger than some $L_0=L_0(\om)$. Thus, \eqref{eq:chem_dist_simple} is proven.

  \textit{(ii)} In order to show the second statement consider
  \begin{align*}
    e_{L} 
    \;\ldef\; 
    \underset{e\in E_d: e^{-} \in [-L,L]^d}{\argmax}\, \om(e).
  \end{align*}
  Then, by construction $\om(e_L)$ is the maximum of order $L^{d-1}$ i.i.d.\ random variables and again by \cite[Theorem~3.5.2]{EKM97} there exists $L_0=L_0(\om)$ such that $\prob$-a.s.
  \begin{align*}
    \om(e_L) 
    \;\geq\; 
    c\, L^{\frac{d-1}{\alpha}}, 
    \qquad \forall\, L \geq L_0.
  \end{align*}
  For such $L$ set $x \ldef e_L^-$ and consider the nearest-neighbour path $(x_{n} : n \in \bbN_0)$ on $\bbZ^d$ defined by $x_0\ldef x$ and $x_{n+1}\ldef x_{n} + \boldsymbol{e}_{i_{n+1}}$ with
  \begin{align*}
    i_{n+1}
    \;\ldef\; 
    \underset{i = 1, \ldots, d}{\argmax}\, 
    \om(\{x_{n}, x_{n} + \boldsymbol{e}_i\}),
    \qquad n \geq 0,
  \end{align*}
  similarly as above in \textit{(i)}.  Then, by setting $y = x_{L}$, we have $d(x, y) = L$ and  
  \begin{align*}
    \om(\{x_{n},x_{n+1}\}) 
    \;\geq\; 
    \om(e_L) 
    \;\geq\;  
    c\, L^{\frac{d-1}{\al}}, 
    \qquad \forall\, n=0, \ldots, L-1.
  \end{align*}
  In particular, \eqref{eq:chem_dist_simple} holds
  %
  %
  and \textit{(ii)} follows from the definition of $d_{\th}^{\om}$.
\end{proof}

\section{Heat kernel upper bounds} \label{sec:UHK}
We work again in the general setting outlined in Section~\ref{sec:prelim} above.
Our main objective is to prove Gaussian-like upper bound on the heat kernel $p_{\th}$ in term of the intrinsic distance $d_{\th}$.   For that purpose, we impose the following assumption on the integrability of the conductances.
\begin{assumption}\label{ass:lim_const}
  Let $d' \geq d \geq 2$. For $p, q, r \in (1, \infty]$ with
  \begin{align}\label{eq:cond_pqr}
    \frac{1}{r} \,+\, \frac{1}{p} \cdot \frac{r-1}{r} \,+\, \frac{1}{q} 
    \;<\; 
    \frac{2}{d'}
  \end{align}
  there exists $C_{\mathrm{int}} \in [1, \infty)$ such that for all $x \in V$ there exists $N_4(x, \om) < \infty$ such that for all $n \geq N_4(x, \om)$,
  \begin{align}
    \Norm{1 \vee \mu^{\om} / \th}{p, B(x,n), \th} \cdot
    \Norm{1 \vee \nu^{\om}}{q, B(x, n)} \cdot
    \Norm{1 \vee \th}{r, B(x,n)} \cdot
    \Norm{1 \vee 1/\th}{q, B(x,n)}
    \;\leq\;
    C_{\mathrm{int}}.
  \end{align}
\end{assumption}
Similarly as explained at the beginning of Section~\ref{sec:example} above, in the context of the random conductance model with ergodic conductances one can use the ergodic theorem to translate  Assumption~\ref{ass:lim_const} directly into a moment condition, provided the speed measure $\th$ is random and stationary, i.e.\ $\th(x) = \th^{\om}(x) = \th^{\tau_x \om}(0)$ for all $x \in \bbZ^d$.

\begin{theorem} \label{thm:hke}
  Suppose that $\om \in \Om$ satisfies Assumption~\ref{ass:lim_const}.  Then, there exist constants $c_i = c_i(d, p, q, C_{\mathrm{int}})$ and  $\gamma = \gamma(d, p, q, C_{\mathrm{int}})$  such that for any given $t$ and $x$ with $\sqrt{t} \geq N_1(x) \vee N_2(x) \vee N_4(x, \om)$ and all $y \in V$ the following hold.
  \begin{enumerate}
  \item [(i)] If $d_{\th}^{\om}(x, y)\leq c_1 t$ then
    \begin{align*}
      p^{\om}_{\th}(t, x, y)
      \;\leq\;
      c_2\, t^{-d/2}\,
      \bigg( 1 + \frac{d(x,y)}{\sqrt{t}} \bigg)^{\!\!\ga}\, 
      \exp\!\bigg( \!-\! c_3\, \frac{d_{\th}^{\om}(x,y)^2}{t}\bigg).
    \end{align*}
  \item [(ii)] If $d_{\th}^{\om}(x,y) \geq c_5 t$ then
    \begin{align*}
      p^{\om}_{\th}(t, x, y)
      \;\leq\;
      c_2\, t^{-d/2}\,  
      \bigg( 1 + \frac{d(x,y)}{\sqrt{t}} \bigg)^{\!\!\ga}\,
      \exp\!\bigg( 
        \!-\! c_4\, d_{\th}^{\om}(x,y) 
        \bigg(1 \vee \log \frac{d_{\th}^{\om}(x,y)}{t}\bigg)
      \bigg).
    \end{align*}
  \end{enumerate}
\end{theorem}
\begin{remark} \label{rem:hke}
  (i) In the case of CSRW or VSRW Theorem~\ref{thm:hke} has been established in \cite{ADS16a}. However, the term $(1+ d(x,y)/\sqrt{t})^{\ga}$ is erroneously missing in the result for the VSRW in \cite[Theorem~1.10]{ADS16a}.  

  (ii) If the distance $d_{\th}^{\om}$ and the graph distance $d$ are comparable,  the estimates in Theorem~\ref{thm:hke} turn into Gaussian upper bounds since then the additional term $(1+ d(x,y)/\sqrt{t})^{\ga}$ can  be absorbed by the exponential term into a constant. Both distances are comparable, for instance, for the CSRW, the VSRW under i.i.d.\ conductances (cf.\ \cite[Lemma~4.2]{BD10}) or for random walks on supercritical percolation clusters with long-range correlations (see \cite{DRS14}).   However, if both distances are not comparable, the bounds in Theorem~\ref{thm:hke} become ineffective in the regime where $d_{\th}^{\om}(x,y) < \sqrt{t} < d(x,y)$, since in this case the term $(1 + d(x,y)/\sqrt{t})^{\ga}$ may become large while the exponential term does not provide a decay yet.  Nevertheless, a near-diagonal bound of the following form can be deduced from the parabolic Harnack inequality established in \cite{ADS16}. There exists $c \in (0, \infty)$ such that for any $x \in V$ and any $t > 0$ with $\sqrt{t} \geq N_1(x) \vee N_2(x) \vee N_4(x, \om)$ (with a slightly modified $N_4$) we have for all $y \in V$,
  \begin{align*}
    p^{\om}_{\th}(t, x, y)
    \;\leq\;
    c\, t^{-d/2},
  \end{align*}
  cf.\ \cite[Proposition~4.7]{ADS16}.

  (iii) The on-diagonal decay $t^{-d/2}$  corresponds to $1/\big|B(x,\sqrt{t})\big|$. In general we expect a stronger decay to hold resulting from the volume of a ball with radius $\sqrt{t}$ w.r.t.\ the distance $d_\theta^\om$  under the speed measure $\theta$. For instance, the heat kernel of the random walk discussed in Section~\ref{sec:example} admits the on-diagonal decay $t^{-(d+1)/2}$, see \cite{DF19}.
\end{remark}
If $x$ and $y$ are sufficiently far apart,  the term $(1+ d(x,y)/\sqrt{t})^{\ga}$ can be simplified  by using Theorem~\ref{thm:dist}.
\begin{coro} \label{cor:hke2}
  Suppose that $\om \in \Omega$ satisfies Assumption~\ref{ass:lim_const}, and assume that there exists $C \geq 1$ and $\ve \in [0, (d-1)/(2p-d+1)]$ such that for any $x \in V$ and all $y \in V$ with $d(x,y) \geq N_3(\om, x)$,
  \begin{align*}
    d(x,y) \;\leq\; C\, d_{\th}^{\om}(x,y)^{1+\ve}.
  \end{align*}
  Then, there exist constants $c_i = c_i(d, p, q, C_{\mathrm{int}})$ such that for any given $t$ and $x$ with $\sqrt{t} \geq N_1(x) \vee N_2(x) \vee N_3(x, \om) \vee N_4(x, \om)$ and all $y \in V$ with $d(x,y) \geq N_3(\om,x)$ the following hold.
  \begin{enumerate}
  \item [(i)] If $d_{\th}^{\om}(x, y)\leq c_1 t$ then
    \begin{align*}
      p^{\om}_{\th}(t, x, y)
      \;\leq\;
      c_6\, t^{-d/2}\,
      \big( 1 \vee d_{\th}^{\om}(x,y)\big)^{\!\ve \ga} \, 
      \exp\!\bigg( \!-\! c_7\, \frac{d_{\th}^{\om}(x,y)^2}{t}\bigg).
    \end{align*}
  \item [(ii)] If $d_{\th}^{\om}(x,y) \geq c_5 t$ then
    \begin{align*}
      p^{\om}_{\th}(t, x, y)
      \;\leq\;
      c_6\, t^{-d/2}\,  
      \big( 1 \vee d_\th^{\om}(x,y) \big)^{\!\ve \ga} \,
      \exp\!\bigg( 
        \!-\! c_8\, d_{\th}^{\om}(x,y) 
        \bigg(1 \vee \log \frac{d_{\th}^{\om}(x,y)}{t}\bigg)
      \bigg).
    \end{align*}
  \end{enumerate}
\end{coro}
\begin{remark}
  Note that $p > (d-1)/2$ for any $p$ satisfying \eqref{eq:cond_pqr}.  Hence, Theorem~\ref{thm:dist} implies that for any $x \in V$ and all $y \in V$ with $d(x,y) \geq N_3(\om,x)$,
  \begin{align*}
    d(x,y)
    \;\leq\;
    c(m_p)^{-2p/(2p-d+1)}\,
    d_{\th}^{\om}(x,y)^{1+\ve}
  \end{align*}
  with $\ve = (d-1)/(2p-d+1)$. In particular, $\ve \to 0$ as $p\to \infty$.
\end{remark}
\begin{proof}
  This is a direct consequence from Theorem~\ref{thm:hke}. Indeed, since
  \begin{align*}
    \bigg( 1 + \frac{d(x,y)}{\sqrt{t}} \bigg)^{\!\!\ga}
    \;\leq\;
    \bigg( 1 + \frac{C \, d_\theta^\om(x,y)^{1+\ve}}{\sqrt{t}} \bigg)^{\!\!\ga} 
    \;\leq\;
    \big( 1 \vee d_{\th}^{\om}(x,y) \big)^{\ve \ga}\,
    \bigg( 1 + C\, \frac{d_\theta^\om(x,y)}{\sqrt{t}} \bigg)^{\!\!\ga},
  \end{align*}
  the second term can be absorbed by the exponential term into a constant.
\end{proof}
In the remainder of this section we explain how the proof of \cite[Theorem~1.6]{ADS16a} needs to be adjusted in order to prove Theorem~\ref{thm:hke}, that is to obtain Gaussian-like upper bounds on the heat kernel for a larger class of speed measures $\theta$. We also take the opportunity to streamline the arguments in \cite{ADS16a} and to correct some technical mistakes leading to the error mentioned in Remark~\ref{rem:hke}.

\subsection{Maximal inequality for the perturbed Cauchy problem}
Consider the following Cauchy problem
\begin{align}\label{eq:cauchy_prob}
   \left\{
    \begin{array}{rcl}
      \partial_t u - \cL_{\th}^{\om} u
      &\mspace{-5mu}=\mspace{-5mu}& 0,
      \\[1ex]
      u(t=0, \,\cdot\,)
      &\mspace{-5mu}=\mspace{-5mu}& f,
    \end{array}
  \right.
\end{align}
for some function $f\!: V \to \bbR$.  Recall that for any given $y \in \bbZ^d$, the function $(t,x) \mapsto p_{\th}^{\om}(t, x, y)$ solves the heat equation \eqref{eq:cauchy_prob} with $f = \indicator_{\{y\}} / \th(y)$.  For any positive function $\phi$ on $V$ such that $\phi, \phi^{-1} \in \ell^{\infty}(V)$ we define the operator $\cL_{\th, \phi}^{\om}$ acting on bounded functions $g\!: V \to \bbR$ as
\begin{align*}
  (\cL_{\th,\phi}^{\om}\, g)(x)
  \;\ldef\;
  \phi(x) (\cL_{\th}^{\om} \phi^{-1} g)(x).
\end{align*}
As a first step we establish the following a-priori estimate.
\begin{lemma} \label{lem:apriori}
  Suppose that $f \in \ell^2(V, \th^{\om})$ and $u$ solves the corresponding Cauchy problem \eqref{eq:cauchy_prob}.  Further, set $v(t,x) \ldef \phi(x) u(t,x)$ for a positive function $\phi$ on $V$ such that $\phi, \phi^{-1} \in \ell^{\infty}(V)$.  Then
  \begin{align}\label{eq:hke:apriori}
    \norm{v(t, \,\cdot\,)}{2}{V, \th}
    \;\leq\;
    \me^{h_{\th}^{\om}(\phi) t}\, \norm{\phi f}{2}{V, \th},
  \end{align}
  where
  \begin{align*}
    h_{\th}^{\om}(\phi)
    \;\ldef\;
    \max_{x \in V} \frac{1}{2 \th(x)}\, \sum_{y \sim x}\, 
    \big| \md \Ga^{\om}(\phi, \phi^{-1})(\{x,y\}) \big|.
  \end{align*}
\end{lemma}
\begin{proof}
  This can be shown by the similar arguments as in \cite[Lemma~2.1]{ADS16a}.
\end{proof}
Our next aim is to derive a maximal inequality for the function $v$. For that purpose we will adapt the arguments given in \cite[Section 4]{ADS16} and set up a Moser iteration scheme.  For any finite interval $I \subset \bbR$, finite, connected $B \subset V$ and $p, p' \in (0, \infty)$, let us introduce a space-time-averaged norm on functions $u\!: \bbR \times V \to \bbR$ by
\begin{align*}
  \Norm{u}{p, p', I \times B, \th}
  \;\ldef\;
  \bigg(
    \frac{1}{|I|}\, \int_{I}\; \Norm{u_t}{p, B, \th}^{p'}\; \md t
  \bigg)^{\!\!1/p'}
  \quad \text{and} \quad
  \Norm{u}{p, \infty, I \times B, \th}
  \;\ldef\;
  \sup_{t \in I} \Norm{u_t}{p, B, \th},
\end{align*}
where $u_t = u(t,.)$, $t \in \bbR$.
\begin{lemma} \label{lem:moser_pre}
  Suppose that $Q = I \times B$, where $I = [s_1, s_2] \subset \bbR$ is an interval and $B \subset V$ is finite and connected.  For a given $\phi > 0$ with $\phi, \phi^{-1} \in \ell^{\infty}(V)$, let $v_t \geq 0$ be a solution of $\partial_t v -  \cL_{\th, \phi}^{\om} v \leq 0$ on $Q$.  Further, let $\eta\!: V \to [0, 1]$ and $\ze\!: \bbR \to [0,1]$ be two cutoff functions with
  \begin{align*}
    \supp \eta \;\subset\; B
    &\qquad \text{and} \qquad
    \eta \;\equiv\; 0 \quad \text{on} \quad \partial B,
    \\
    \supp \ze \;\subset\; I
    &\qquad \text{and} \qquad
    \ze(s_1) \;=\; 0.
  \end{align*}
  Then, there exists $C_1 < \infty$ such that for $\al \geq 1$ and $p, p_* \in (1, \infty)$ with $1/p + 1/p_* = 1$, 
  \begin{align}\label{eq:energy:est:Moser}
    &\frac{1}{|I|}\, 
    \Norm{\ze (\eta v^{\al})^2}{1, \infty, Q, \th}
    \,+\, 
    \frac{1}{|I|}\,
    \int_{I} \ze(t)\; \frac{\cE^{\om}(\eta v_t^{\al})}{|B|} \, \md t
    \nonumber\\[1ex]
    &\leq\;
    C_1 \al^2\, 
    \Big(
      \Norm{\mu^\om / \th}{p, B, \th}\, \norm{\nabla \eta}{\infty\!}{\!E}^2\,
      \Norm{v^{2 \al}}{p_*, 1, Q, \th}
      +
      \Big(
        \big\| 
          \ze' 
        \big\|_{\raisebox{-0ex}{$\scriptstyle L^{\raisebox{.1ex}{$\scriptscriptstyle \!\infty$}} (I)$}}
        +
        h_{\th}^{\om}(\phi)
      \Big)
      \Norm{v^{2 \al}}{1, 1, Q, \th}
    \Big).
  \end{align}
\end{lemma}
\begin{proof}
  Fix some $\al \geq 1$.  Since $v \geq 0$ satisfies $\partial_t v - \cL_{\th, \phi}^{\om} v \leq 0$ on $Q$, a summation by parts yields
  \begin{align}\label{eq:time:deriv}
    \frac{1}{2 \al}\, \partial_t \norm{\eta v_t^{\al}}{2}{V, \th}^2
    \;\leq\;
    - \scpr{\nabla(\eta^2 \phi v_t^{2\al-1})}{\om \nabla(\phi^{-1} v_t)}{E}
  \end{align}
  for any $t \in I$. By applying the product rule \eqref{eq:rule:prod}, we obtain
  \begin{align}\label{eq:T1+T2:def}
    &\scpr{\nabla(\eta^2 \phi v_t^{2\al-1})}{\om \nabla(\phi^{-1} v_t)}{E}
    \nonumber\\[.5ex]
    &\mspace{36mu}\;=\;
    \scpr{\av{\eta^2}}{\md \Ga^{\om}(\phi v_t^{2\al-1}\!, \phi^{-1} v_t)}{E}
    \,+\,
    \scpr{\av{\phi v_t^{2\al-1}}}{\md \Ga^{\om}(\eta^2\!, \phi^{-1} v_t)}{E}
    \nonumber\\[.5ex]
    &\mspace{36mu}\;\rdef\;
    T_1 \,+\, T_2. 
  \end{align}
  Let us first focus on the term $T_1$.  Again, an application of the product rule \eqref{eq:rule:prod} together with the fact that $(\nabla \phi)(\nabla \phi^{-1}) \leq 0$ and $ -\av{\phi^{-1}} (\nabla \phi) = \av{\phi} (\nabla \phi^{-1})$, yields the following lower bound
  \begin{align*}
    \md \Ga^{\om}(\phi v_t^{2\al-1}\!, \phi^{-1} v_t)
    &\;\geq\;
    \av{\phi} \av{\phi^{-1}}\, 
    \md \Ga^{\om}(v_t^{2\al-1}\!, v_t)
    \,+\,
    \av{v_t^{2\al}}\, \md \Ga^{\om}(\phi, \phi^{-1})
    \\
    &\mspace{29mu}+\,
    \av{\phi}
    \Big(
      \av{v_t}\, \md \Ga^{\om}(v_t^{2\al-1}, \phi^{-1}) 
      - 
      \av{v_t^{2\al-1}}\, \md \Ga^{\om}(v_t, \phi^{-1})
    \Big),
  \end{align*}
  where we used that by H\"older's inequality, $\av{v_t^{\al_1}} \av{v_t^{\al_2}} \leq \av{v_t^{\al_1+\al_2}}$ for any $\al_1, \al_2 \geq 0$.  Further, by \cite[Lemma~B.1]{ADS16a}, we have that
  \begin{align*} 
    \md \Ga^{\om}(v_t^{2\al-1}, v_t)
    &\;\geq \;
    \frac{2\al-1}{\al^{2}}\, \md\Ga^{\om}(v_t^{\al}\!, v_t^{\al}), 
  \end{align*}
  and
  \begin{align}\label{eq:est:v}
    &\big|
      \av{v_t}(e) \nabla v_t^{2\al-1}(e) - \av{v_t^{2\al-1}}(e) \nabla v_t(e)
    \big|
    \nonumber\\
    &\mspace{36mu}=\;
    \big|v_t^{2\al -1}(e^+) v_t(e^-) - v_t^{2\al-1}(e^-) v_t(e^+) \big|
    \;\leq\;
    \frac{2(\al-1)}{\al}\, \big| \av{v_t^{\al}}(e)\, \nabla v_t^{\al}(e) \big|
  \end{align}
  for all $e \in E$. Thus, by combining the estimates above and using that
  \begin{align}\label{eq:av+phi:sqrt}
    \av{\phi} \big|\nabla \phi^{-1}\big| 
    \;=\; 
    \sqrt{ \av{\phi} \av{\phi^{-1}} } 
    \cdot \sqrt{-(\nabla \phi)(\nabla \phi^{-1}) },
  \end{align}
  an application of Young's inequality, that reads $|a b| \leq \frac{1}{2}(\ve a^2 + b^2/\ve)$, with $\ve = 1/(2\al)$ results in
  \begin{align*}
    T_1
    &\;\geq\;
    \frac{3\al-1}{2\al^2} 
    \scpr{\av{\eta^2} \av{\phi} \av{\phi^{-1}}}
    {\md \Ga^{\om}(v_t^{\al}\!, v_t^{\al})}{E}
    \,-\,
    2\al\, |B|\, h_{\th}^{\om}(\phi)\, \Norm{v_t^{2\al}}{1, B, \th}.
  \end{align*}
  Let us now address the term $T_2$.  Observe that
  \begin{align*}
    &\av{\phi v_t^{2\al-1}}\,
    \md \Ga^{\om}(\phi^{-1} v_t, \eta^2)
    \\
    &\mspace{36mu}=\;
    2 \av{\phi v_t^{2\al-1}} \av{\eta}
    \Big( 
      \av{\phi^{-1}} \, \md \Ga^{\om}(v_t, \eta)
      \,+\,
      \av{v_t}\, \md \Ga^{\om}(\phi^{-1}, \eta)
    \Big)
    \\
    &\mspace{36mu}\geq\;
    -4 \av{\eta} \av{\phi} \av{v_t^{2\al-1}}\,
    \Big( 
      \av{\phi^{-1}} \, \big|\md \Ga^{\om}(v_t, \eta)\big|
      \,+\,
      \av{v_t}\, \big|\md \Ga^{\om}(\phi^{-1}, \eta)\big|
    \Big).
  \end{align*}
  Since $\av{v_t^{2\al-1}}\av{v_t} \leq \av{v_t^{2\al}}$, an application of the Young inequality yields
  \begin{align*}
    &4 \av{\eta} \av{\phi} \av{v_t^{2\al}}\,
    \big|\md \Ga^{\om}(\phi^{-1}, \eta)\big|
    \\[.5ex]
    &\mspace{36mu}\overset{\!\!\!\eqref{eq:av+phi:sqrt}\!\!\!}{\leq\;}
    8\,\av{\phi} \av{\phi^{-1}} \av{v_t^{2\al}}\,
    \md \Ga^{\om}(\eta, \eta)
    \,-\,
    \frac{1}{2} \av{\eta^2} \av{v_t^{2\al}}\, 
    \md \Ga^{\om}(\phi, \phi^{-1}).
  \end{align*}
  On the other hand,
  \begin{align*}
    &\big| \av{v_t^{2\al-1}}(e) (\nabla v_t)(e) \big|
    \\
    &\mspace{36mu}\leq\;
    \Big|
      \av{v_t^{\al}}(e) (\nabla v_t^{\al})(e)
     \Big| \,+\,
      \frac{1}{2}\, \Big|
      \big(v_t^{2\al -1}(e^+) v_t(e^-) - v_t^{2\al-1}(e^-) v_t(e^+) \big)
    \Big|
    \\
    &\mspace{36mu}\overset{\!\!\!\eqref{eq:est:v}\!\!\!}{\leq\;}
    \frac{2\al-1}{\al}\,
    \big| \av{v_t^{\al}}(e)\, \nabla v_t^{\al}(e) \big|.
  \end{align*}
  Thus, by applying again Young's inequality with $\ve = 1/(4\al)$, we get
  \begin{align*}
    &4 \av{\eta} \av{\phi} \av{\phi^{-1}} \av{v_t^{2\al-1}}\, 
    \big|\md \Ga^{\om}(v_t, \eta)\big|
    \\[1ex]
    &\mspace{36mu}\leq\;
    4\, \frac{2\al-1}{\al}\, 
    \av{\phi} \av{\phi^{-1}} \av{\eta} \av{v_t^{\al}}\,
    \big| \md \Ga^{\om}(v_t^{\al}, \eta) \big|
    \\
    &\mspace{36mu}\leq\;
    \av{\phi} \av{\phi^{-1}}
    \bigg(
      \frac{2\al-1}{2\al^2}\, \av{\eta^2}\, 
      \md \Ga^{\om}(v_t^{\al}\!, v_t^{\al})
      \,+\,
      8(2\al-1)\, \av{v_t^{2\al}}\,
      \md \Ga^{\om}(\eta, \eta)
    \bigg)
  \end{align*}
  Hence, the estimates above together with the fact that
  \begin{align*}
    \av{\phi^{-1}} \av{\phi}
    \;=\;
    1 -
    \dfrac{1}{4}\, (\nabla \phi) (\nabla \phi^{-1})
  \end{align*}
  give rise to the following lower bound
  \begin{align*}
    T_2
    &\;\geq\;
    -\frac{2\al-1}{2\al^2} 
    \scpr{\av{\eta^2} \av{\phi} \av{\phi^{-1}}}
    {\md \Ga^{\om}(v_t^{\al}\!, v_t^{\al})}{E}
    \nonumber\\[.5ex]
    &\mspace{30mu}-\,
    16 \al\, |B|\, \Norm{\mu^{\om} / \th}{p, B, \th}\, 
    \norm{\nabla \eta}{\infty}{E}^2\,
    \Norm{v_t^{2\al}}{p_*, B, \th}
    \,-\,
    5 \al\, h_{\th}^{\om}(\phi)\, |B|\, \Norm{v_t^{2\al}}{1, B, \th}.
  \end{align*}
  Since $\av{\phi} \av{\phi^{-1}} \geq 1$ and 
  \begin{align*}
    \av{\eta^2}\, \md \Ga^{\om}(v_t^{\al}\!, v_t^{\al})
    \;\geq\; \frac 1 2 \,
    \md \Ga^{\om}(\eta v_t^{\al}\!, \eta v_t^{\al})
    \,-\,
    \av{v_t^{2\al}}\, \md \Ga^{\om}(\eta, \eta),
  \end{align*}
  we obtain that there exists $C_1 < \infty$ such that
  \begin{align*}
    T_1 + T_2
    &\;\geq\;
    \frac{1}{4\al}\, \cE^{\om}(\eta v_t^{\al})
    \\
    &\mspace{31mu}-\,
    \frac{C_1}{4}\, \al\, |B|\,  
    \Big(
      \Norm{\mu^{\om}/\th}{p, B, \th}\, \norm{\nabla \eta}{\infty}{E}^2\,
      \Norm{v_t^{2\al}}{p_*, B, \th}
      \,+\,
      h_{\th}^{\om}(\phi)\, \Norm{v_t^{2\al}}{1, B, \th}
    \Big).
  \end{align*}
  Hence, 
  \begin{align}\label{eq:energy:est:deriv}
    &2\, \partial_t \Norm{(\eta v_t^{\al})^2}{1, B}
    \,+\,
    \frac{\cE_t^{\om}(\eta v_t^{\al})}{|B|}
    \nonumber\\[.5ex]
    &\mspace{36mu}\leq\;
    C_1 \al^2\, 
    \Big(
      \Norm{\mu^{\om}/\th}{p, B, \th}\, \norm{\nabla \eta}{\infty}{E}^2\,
      \Norm{v_t^{2\al}}{p_*, B, \th}
      \,+\,
      h^{\om}(\phi)\, \Norm{v_t^{2\al}}{1, B, \th} 
    \Big).
  \end{align}
  Finally, since $\ze(s_1) = 0$,
  \begin{align*}
    \int_{s_1}^{s}\!\!
      \ze(t)\, \partial_t \Norm{(\eta v_t^{\al})^2}{1, B}\;
    \md t
    &\;=\;
    \int_{s_1}^{s}
      \Big(
        \partial_t\big( \ze(t) \Norm{(\eta v_t^{\al})^2}{1, B} \big)
        -
        \ze'(t) \Norm{(\eta v_t^{\al})^2}{1,B}
      \Big)\;
    \md t
    \\[.5ex]
    &\;\geq\;
    \frac{\ze(s)}{2} \Norm{(\eta v_s^{\al})^2}{1, B}
    \,-\,
    \|\ze'\|_{\raisebox{-.0ex}{$\scriptstyle L^{\raisebox{.2ex}{$\scriptscriptstyle {\!\infty\!}$}} (I)$}}\, 
    |I|\, \Norm{v^{2\al}}{p_*, 1, Q}
  \end{align*}
  for any $s \in (s_1, s_2]$.  Thus, by multiplying both sides of \eqref{eq:energy:est:deriv} with $\ze(t)$ and integrating the resulting inequality over $[s_1, s]$ for any $s \in I$, the assertion \eqref{eq:energy:est:Moser} follows by an application of the H\"older inequality.
\end{proof}
For any $x_0 \in V$, $\de \in (0,1)$ and $n \geq 1$, we write $Q_{\de}(n) \equiv [0,  \de n^2] \times B(x_0, n)$ to denote the corresponding space-time cylinder, and we set
\begin{align*}
  Q_{\de, \si}(n)
  \;\ldef\;
  \big[( 1 - \si) s', (1-\si) s'' + \si \de n^2 \big] \times B(x_0, \si n),
  \qquad \si \in (0, 1],
\end{align*}
where $s' = \ve \de n^2$ and $s'' = (1-\ve)\de n^2$ for some fixed $\ve \in (0, 1/4)$.
\begin{prop}
  For $x_0 \in V$, $\de \in (0, 1]$,  $\ve \in (0, 1/4)$ and $n \geq N_1(x_0) \vee N_2(x_0)$, let $v > 0$ be such that $\partial_t v - \cL_{\th, \phi}^{\om} v = 0$ on $Q(x_0, n)$.  Then, for any $p, q, r \in (1, \infty]$ satisfying \eqref{eq:cond_pqr} there exists $C_2 \equiv C_2(d, p, q, r, \varepsilon) < \infty$ and $\ka = \ka(d', p, q, r) < \infty$ such that
  \begin{align} \label{eq:max_ineqV}
    \max_{(t,x) \in  Q_{\de, 1/2}(n)} v(t,x)
    \;\leq\;
    \frac{C_2}{n^{d/2}}\,  \bigg(\frac{m^{\om}(n)}{\varepsilon\de}\bigg)^{\!\ka}
    \me^{2 (1 - \ve) h(\phi) \de n^2}\,
    \norm{\phi f}{2}{V, \th},
  \end{align}
  where
  \begin{align*}
    m^{\om}(n) 
    \;\ldef\; 
    \Norm{1 \vee \frac{\mu^{\om}}{\th}}{p, B(x_0, n), \th} \cdot 
    \Norm{1 \vee \nu^{\om}}{q, B(x_0, n)} \cdot
    \Norm{1 \vee \th}{r, B(x_0, n)} \cdot 
    \Norm{1 \vee \frac{1}{\th}}{q, B(x_0,n)}.
  \end{align*}
\end{prop}
\begin{proof}
  We will follow similar arguments as in the proof of \cite[Proposition 4.2]{ADS16}.  Fix some $1/2 \leq \si' < \si \leq 1$.  For $p, r \in (1, \infty)$, let $p_* \ldef p/(p-1)$ and $r_* \ldef r / (r-1)$ be the H\"older conjugate of $p$ and $r$, respectively.  For any $k \in \bbN_0$ set $\al_k \ldef \al^k$, where
  \begin{align*}
    \al
    \;\ldef\;
    1 + \frac{1}{p_*} - \frac{r_*}{\rho}
    \qquad \text{and} \qquad
    \rho
    \;\ldef\;
    \frac{d'}{d'-2+d'/q}.
  \end{align*}
  Notice that for any $p, q, r \in (1, \infty)$ satisfying \eqref{eq:cond_pqr} we have $\al > 1$.  In particular,  $r_* / \rho + 1/p < 1$.  Further, for 
  \begin{align*}
    \si_k \;=\; \si' + 2^{-k} (\si - \si')
    \qquad \text{and} \qquad
    \tau_k \;=\; 2^{-k-1} (\si - \si'),
    \quad k \in \bbN_0,
  \end{align*}
  we write $I_k \ldef [(1-\si_k)s', (1-\si_k)s''+\si_k \de n^2]$, $B_k \ldef B(x_0, \si_k n)$ and $Q_k \ldef Q_{\de, \si_k}(n)$ to lighten notation.  Note that $|I_k| / |I_{k+1}| \leq 2$ and $|B_k| / |B_{k+1}| \leq 2^d C_{\mathrm{reg}}^2$.  The constant $c \in (0, \infty)$ appearing in the computations below is independent of $n$ but may change from line to line. 
First, by using H\"older's and Young's inequality, 
  \begin{align} \label{eq:norm_interpol}
    \Norm{v^{2\al_k}}{\al p_*, \al, Q_{k+1}, \th}
    &\;\leq\;
    \Norm{v^{2\al_k}}{1, \infty, Q_{k+1}, \th} 
    \,+\,
    \Norm{v^{2\al_k}}{\rho / r_*, 1, Q_{k+1}, \th}
  \end{align}
  (cf.\ \cite[Lemma 1.1]{KK77}).   Due to the discrete structure of the underlying space $\bbZ^d$, the discrete balls $B_{k+1}$ and $B_k$ may coincide whenever $\tau_k n$ is sufficiently small.  For this reason, we proceed by distinguishing two different cases.  
  
  First consider the case $\tau_k n \geq 1$. For any $k \in \bbN_0$ let $\eta_k$ be a cut-off functions in space and $\ze_k \in C^{\infty}(\bbR)$ be a cut-off function in time such that $\supp \eta_k \subset B_{k}$, $\eta_k \equiv 1$ on $B_{k+1}$, $\eta_k \equiv 0$ on $\partial B_{k}$, $\norm{\nabla \eta_k}{\infty}{E} \leq 1/(\tau_{k} n)$ and $\supp \ze_k \subset I_{k}$, $\ze_k \equiv 1$ on $I_{k+1}$, $\ze_k( (1-\si_k)s') = 0$ and $\big\| \ze_k' \big\|_{\raisebox{-0ex}{$\scriptstyle L^{\raisebox{.1ex}{$\scriptscriptstyle \!\infty$}} ([0,\de n^2])$}} \leq  1 / (\tau_k \de n^2)$. Then,  from \eqref{eq:norm_interpol} we get
  \begin{align*}
    \Norm{v^{2\al_k}}{\al p_*, \al, Q_{k+1}, \th}
    &\;\leq\;
    c\, 
    \Big(
      \Norm{\ze_k(\eta_k v^{\al_k})^2}{1, \infty, Q_k, \th}
      \,+\,
      \Norm{\ze_k(\eta_k v^{\al_k})^2}{\rho / r_*, 1, Q_k, \th}
    \Big).
  \end{align*}
  Further, by Assumption~\ref{ass:graph}(iii) we may apply the Sobolev inequality for functions with compact support in \cite[Equation~(28)]{ADS15} to obtain
  \begin{align*}
    \Norm{\ze_k(\eta_k v^{\al_k})^2}{\rho / r_*, 1, Q_k, \th}
    &\;\leq\;
    c\, n^2\, \Norm{\nu^{\om}}{q, B_k}\, 
    \Norm{\th}{r, B_k}^{r_*/\rho}\;
    \frac{1}{|I_k|}\,
    \int_{I_k} 
      \ze_k(t) \frac{\cE^{\om}(\eta_k v_t^{\al_k})}{|B_k|}
    \md t.
  \end{align*}
  Hence,
  \begin{align}
    &\Norm{\ze_k(\eta_k v^{\al_k})^2}{1, \infty, Q_k, \th}
    \,+\,
    \Norm{\ze_k(\eta_k v^{\al_k})^2}{\rho / r_*, 1, Q_k, \th}
    \nonumber\\[1ex]
    &\mspace{36mu}\leq\;
    c\, n^2\, 
    \bigg(
      \frac{1}{|I_k|}\Norm{\ze_k(\eta_k v^{\al_k})^2}{1, \infty, Q_k, \th}
      \,+\,
      \frac{\Norm{\nu^{\om}}{q, B_k}\,\Norm{\th}{r, B_k}^{r_*/\rho}}{|I_k|}
      \int_{I_k} \ze_k(t) \frac{\cE^{\om}(\eta_k v^{\al_k})}{|B_k|}\, \md t
    \bigg)
    \nonumber\\[1ex]
    &\mspace{36mu}\overset{\!\!\!\!\eqref{eq:energy:est:Moser}\!\!\!\!}{\leq\;}
    c\, \al_k^2\, 
    \frac{m^{\om}(n)}{\Norm{1 \vee 1/\th}{1, B_k}}\,
    \bigg(
      \frac{1}{\de \tau_k^2}
      \,+\,
      n^2 h_{\th}^{\om}(\phi)
    \bigg)\,
    \Norm{v^{2\al_k}}{p_*, 1, Q_k, \th}.
    \label{eq:Moser:iter:step}
  \end{align}
  Thus, by combining the estimates above, we get
  \begin{align} \label{eq:Moser:iter:step2}
    &\Norm{v}{2\al_{k+1}p_*, 2\al_{k+1}, Q_{k+1}, \th} 
    \;=\;
    \Norm{v^{2\al_k}}{\al p_*, \al, Q_{k+1}, \th}^{1/(2\al_k)}
    \nonumber
    \\[1ex]
    &\mspace{36mu}\leq\;
    \bigg(
      c\, 2^{2k} \al_k^2\, 
      \frac{\big(1 + \de n^2 h_{\th}^{\om}(\phi)\big)}{\de( \si - \si' )^2}\,
      \frac{m^{\om}(n)}{\Norm{1 \vee 1/\th}{1, B_k}}
    \bigg)^{\!\!1/(2\al_k)}
    \Norm{v}{2\al_kp_*, 2\al_k, Q_k, \th}.
  \end{align}

  Next we consider the case $\tau_k n < 1$. Again we shall estimate both terms on the right hand side of \eqref{eq:norm_interpol}.  Note that for any $t \in I_k$,
  \begin{align*}
    \Norm{v_t^{2 \al_k}}{\rho/r_*, B_{k+1}, \th}
    &\;\leq\;
    \Big( \max_{x \in B_k} v_t(x)^{2 \al_k} \Big)^{\!1-r_*p_*/\rho}\,
    \Norm{v_t^{2 \al_k}}{p_*, B_k, \th}^{r_* p_*/\rho}
    \\[.5ex]
    &\;\leq\;
    \Big(
      |B_k|^{(1+1/q)/p_*}\, \Norm{v_t^{2 \al_k}}{p_*/(1+1/q), B_k}
    \Big)^{\!1-r_*p_*/\rho}\,
    \Norm{v_t^{2 \al_k}}{p_*, B_k,\th}^{r_* p_*/\rho}
    \\[.5ex]
    &\;\leq\;
    |B_k|^{(1+1/q) (1/p_*-r_*/\rho)}\, \Norm{1/\th}{q,B_k}^{1/p_*-r_*/\rho}\,
    \Norm{v_t^{2 \al_k}}{p_*, B_k, \th},   
  \end{align*}
  where we have used that by H\"older's inequality
  \begin{align*}
    \Norm{v_t^{2 \al_k}}{p_*/(1+1/q), B_k}
    \;\leq\;
    \Norm{v_t^{2 \al_k} \th^{1/p_*}}{p_*, B_k}\,
    \Norm{\th^{-1/p_*}}{p_*q,B_k}
    \;\leq \;
    \Norm{v_t^{2 \al_k}}{p_*, B_k, \th} \, \Norm{1/\th}{q,B_k}^{1/p_*}.
  \end{align*}      
  Since $d (1 + 1/q) (1/p_* - r_*/\rho) \leq 2$ and $n < 1 / \tau_k$, we find
  \begin{align}\label{eq:aprio_term1}
    \Norm{v^{2\al_k}}{\rho/r_*, 1, Q_{k+1}, \th}
    \;\leq\;
    c\, \frac{2^{2k}}{(\si - \si')^{2}}\, \Norm{1/\th}{q,B_k}^{1/p_*-r_*/\rho}\,
    \Norm{v^{2 \al_k}} {p_*, 1, Q_k, \th}.
  \end{align}
  In order to estimate the first term on the right hand side of \eqref{eq:norm_interpol}, recall that $v_t(x) = \phi(x) u_t(x)$, where
  \begin{align*}
    \partial_t u_t(x)
    \;=\;
    (\cL^{\om}_{\th} u_t)(x)
    \;\geq\;
    - \frac{\mu^\om(x)}{\th(x)}\, u_t(x).
  \end{align*}
  Hence, $ \partial_t v_t(x) \th(x) \geq - \mu^{\om}(x) v_t(x)$ and therefore
  \begin{align} \label{eq:partialNormv_t}
    \frac{1}{2\al_k} \, \partial_t \Norm{v_t^{2 \al_k}}{1, B_k, \th}
    \;\geq\;
    -\Norm{v_t^{2 \al_k}}{1, B_k, \mu^\om}.
  \end{align}
  Let now $\xi_k \in C^{\infty}(\bbR)$ be a cut-off function in time such that $\supp \xi_k \subset I_{k}$, $\xi_k \equiv 1$ on $I_{k+1}$, $\xi_k( t_k ) = 0$ and $\big\| \xi_k' \big\|_{\raisebox{-0ex}{$\scriptstyle L^{\raisebox{.1ex}{$\scriptscriptstyle \!\infty$}} ([0,\de n^2])$}} \leq  1 / (\varepsilon \tau_k \de n^2)$, where we write in short $t_k\ldef (1-\si_k)s''+\si_k \de n^2$ for the right endpoint of $I_k$.  We also choose $t_* \in I_{k+1}$ such that
  \begin{align*}
    \Norm{v_{t_*}^{2 \al_k}}{1, B_k, \th}
    \;=\; 
    \max_{t \in I_{k+1}} \Norm{v_t^{2 \al_k}}{1, B_k, \th}. 
  \end{align*}
  Then, from \eqref{eq:partialNormv_t} and product rule we get
  \begin{align*}
    \partial_t \Big( \xi_k(t) \, \Norm{v_t^{2 \al_k}}{1,B_k,\th} \Big)
    \;\geq\;
    \xi_k'(t)\, \Norm{v_t^{2 \al_k}}{1, B_k, \th}
    \,-\,  2\al_k\, \xi_k(t)\, \Norm{v_t^{2 \al_k}}{1, B_k, \mu^\om},
  \end{align*}
  and an integration over $t$ yields
  \begin{align*}
    \max_{t\in I_{k+1}} \Norm{v_t^{2 \al_k}}{1, B_k, \th}
    \;\leq\;
    \int_{t_*}^{t_k}
      \Big(
        2\al_k\, \xi_k(t)\, \Norm{v_t^{2 \al_k}}{1, B_k, \mu^\om}
        \,-\, \xi_k'(t)\, \Norm{v_t^{2 \al_k}}{1, B_k, \th}
      \Big)\,
    \md t,
  \end{align*}
  so that
  \begin{align}\label{eq:aprio_term2}
    \Norm{v^{2\al_k}}{1, \infty, Q_{k+1}, \th}
    &\;\leq\;
    2 \al_k\, |I_k|\, \Norm{v^{2\al_k}}{1, 1, Q_{k}, \mu^\om}
    \,+\, \frac{1}{\ve \tau_k} \, \Norm{v^{2\al_k}}{1, 1, Q_{k}, \th}
    \nonumber \\
    &\;\leq\;
    \frac{c \de \al_k 2^k}{\ve (\si -\si')}\,
    \Norm{1 \vee \frac{\mu^{\om}}{\th}}{p, B_k, \th}\,
    \Norm{v^{2\al_k}}{p_*, 1, Q_{k}, \th},
  \end{align}
  where we used Jensen's and H\"older's inequalities in the last step.  Combining \eqref{eq:norm_interpol} with \eqref{eq:aprio_term1} and \eqref{eq:aprio_term2} we get in the case $\tau_k n < 1$,
  \begin{align} \label{eq:Moser:iter:step3}
    \Norm{v}{2\al_{k+1}p_*, 2\al_{k+1}, Q_{k+1}, \th} 
    &\;=\;
    \Norm{v^{2\al_k}}{\al p_*, \al, Q_{k+1}, \th}^{1/(2\al_k)}
    \nonumber\\[.5ex]  
    &\;\leq\;
    \bigg(
      \frac{ c\, 2^{2k} \al_k}{\ve ( \si - \si' )^2}\, m^{\om}(n)
    \bigg)^{\!\!1/(2\al_k)}\,
    \Norm{v}{2\al_kp_*, 2\al_k, Q_k, \th}.
  \end{align}    
  By iterating \eqref{eq:Moser:iter:step2} and \eqref{eq:Moser:iter:step3}, respectively, and using the fact that $\sum_{k=0}^{\infty} k/\al_k < \infty$, there exists $c < \infty$ independent of $K$ such that
  \begin{align*}
    \Norm{v}{2\al_{K}p_*, 2\al_{K}, Q_{K}, \th}
    \;\leq\;
    c\,
    \prod_{k=0}^{K-1}
    \bigg(
      \big(1 + \de n^2 h_{\th}^{\om}(\phi)\big)\,
      \frac{m^{\om}(n)}{\ve \de (\si-\si')^2}
    \bigg)^{\!\!1/(2 \al_k)}\,
    \Norm{v}{2p_*, 2, Q_{\de, \si}(n), \th}.
  \end{align*}
  Setting $\ka / p_* \ldef \frac{1}{2} \sum_{k=0}^{\infty} 1/\al_k$ and using that $Q_K \downarrow Q_{\de, 1/2}(n)$ we get
  \begin{align*}
    \max_{(t,x) \in Q_{\de, 1/2}(n)} v(t,x)
    &\;=\;
    \lim_{K \to \infty} \Norm{v}{2\al_{K}p_*, 2\al_{K}, Q_{K}, \th}
    \\[.5ex]
    &\;\leq\;
    c\,
    \bigg(
      \big(1 + \de n^2 h_{\th}^{\om}(\phi)\big)\,
      \frac{m^{\om}(n)}{\ve \de (\si-\si')^2}
    \bigg)^{\!\!\ka/ p_*}\,
    \Norm{v}{2p_*, 2, Q_{\de, \si}(n), \th}.
  \end{align*}
  Finally, by using similar arguments as in \cite[Theorem~2.2.3]{S-C02} or \cite[Corollary~3.9]{ADS15}, there exists $c \equiv c(p, q, r, d') < \infty$ such that
  \begin{align*}
    \max_{(t,x) \in Q_{\de, 1/2}(n)} v(t,x)
    &\;\leq\;
    c\,
    \bigg(
      \Big(1 + \de n^2 h_{\th}^{\om}(\phi)\Big)\,
      \frac{m^{\om}(n)}{\varepsilon \de}
    \bigg)^{\!\!\ka}\,
    \Norm{v}{2, \infty, Q_{\de}(n), \th}
    \\[1ex]
    &\overset{\!\!\eqref{eq:hke:apriori}\!\!}{\;\leq\;}
    \frac{c\, C_{\mathrm{reg}}^{1/2}}{n^{d/2}}\, 
    \bigg(
      \Big(1 + \de n^2 h_{\th}^{\om}(\phi)\Big)\,
      \frac{m^{\om}(n)}{\varepsilon \de}
    \bigg)^{\!\!\ka}\,
    \me^{h_{\th}^{\om}(\phi) \de n^2}\, \norm{\phi f}{2}{V, \th}.
  \end{align*}
  Since for any $\ve \in (0, 1/2)$ there exists $c(\ve) < \infty$ such that for all $n \geq 1$ and $\de \in (0,1]$,
  \begin{align*}
    \Big(1 + \de n^2 h_{\th}^{\om}(\phi)\Big)^{\!\ka}\,
    \me^{-(1-2\ve) h_{\th}^{\om}(\phi) \de n^2}
    \;\leq\;
    c(\ve)
    \;<\; 
    \infty,
  \end{align*}
  the claim follows.
\end{proof}

\subsection{Heat kernel bounds}
\begin{prop}  \label{prop:est_cp}
  Suppose that Assumption~\ref{ass:lim_const} hold and let $x_0 \in V$ be fixed. Then,  for any given $x \in V$ and $t$ with $\sqrt{t} \geq  N_1(x_0) \vee N_2(x_0) \vee N_3(x_0, \om)$ the solution $u$ of the Cauchy problem in \eqref{eq:cauchy_prob} satisfies
  \begin{align*}
    |u(t,x)|
    \;\leq\;
    C_3\,t^{-d/2}\,
    \sum_{y\in V}\! \bigg(1 + \frac{d(x_0,x)}{\sqrt{t}}\bigg)^{\!\!\ga}
    \bigg( 1 + \frac{d(x_0,y)}{\sqrt{t}}\bigg)^{\!\!\ga}\;
    \frac{\phi(y)}{\phi(x)}\, \me^{2 h_{\th}^{\om}(\phi)t}\, f(y)
 \end{align*}
 with $\ga \ldef 2\ka - d/2$ and $C_3= C_3(d, p, q, C_{\mathrm{int}})$.
\end{prop}
\begin{proof}
  Given \eqref{eq:max_ineqV} this follows as in the proof of \cite[Proposition~2.7]{ADS16a}. 
\end{proof}
\begin{proof}[Proof of Theorem~\ref{thm:hke}]
  First, notice that the heat kernel $(t, x) \mapsto p^{\om}_\theta(t, x, y)$ solves the Cauchy problem \eqref{eq:cauchy_prob} with $f = \indicator_{\{y\}}/\th(y)$.  Further, let $x_0 \in V$ be arbitrary but fixed and consider the function $\phi = \me^{\psi}$ with $\psi(z) \ldef - \la \min\big\{ d_{\th}^{\om}(x, z), d_{\th}^{\om}(x, y) \big\}$ for $\la > 0$.  Then, for sufficiently large $t$, an application of Proposition~\ref{prop:est_cp} yields
  \begin{align*}
    p^{\om}_\theta(t, x, y)
    \;\leq\;
    C_3\, t^{-d/2}\,
    \bigg( 1 + \frac{d(x_0, x)}{\sqrt t} \bigg)^{\!\!\ga}
    \bigg( 1 + \frac{d(x_0, y)}{\sqrt t} \bigg)^{\!\!\ga}\,
    \me^{\psi(y) - \psi(x) + 2 h_\th^\om(\phi) t}.
  \end{align*}
Next we optimise over $\la > 0$.  Since
  \begin{align*}
    \big|\nabla \psi(e)\big|
    \;\leq\;
    \la\, \big|d_{\th}^{\om}(x,e^+) - d_{\th}^{\om}(x,e^-)\big|
    \overset{\eqref{eq:def_chemdist}}{\;\leq\;} \lambda \,
    \bigg(
      1 \wedge \frac{\th(e^+) \wedge \th(e^-)}{\om(e)}
    \bigg)^{\!\!1/2}
  \end{align*}
  and $a \big(\cosh(x)-1\big)\leq \cosh(\sqrt{a} x)-1$ for all $x\in \bbR$  and any $a\geq 1$, we get
  \begin{align*}
    h_\th^\om(\phi)
    &\;= \;
    \max_{x\in V}  \sum_{y\sim x} \frac{ \om(x,y)}{\theta(x)} \Big(\cosh\big(\nabla \psi(\{x,y\})\big)-1 \Big)
    \;\leq\;
    C_{\mathrm{deg}} \big( \cosh(\la)-1 \big).
  \end{align*} 
  Hence,
  \begin{align*}
    \exp\!\big(\psi(y) - \psi(x) + 2h_{\th}^{\om}(\phi) t\big)
    &\;\leq \;
    \exp\!\bigg(
      d_{\th}^{\om}(x, y) 
      \Big( 
        \!-\!
        \la + \frac{2C_{\mathrm{deg}}t}{d_{\th}^{\om}(x, y)}\, 
        \big( \cosh(\la) - 1 \big)
      \Big)
    \bigg).
  \end{align*}
  By setting
  \begin{align*}
    F(s)
    \;=\;
    \inf_{\la >0} \Big( \!-\!\la + s^{-1} \big( \cosh(\la)-1\big) \Big),
  \end{align*}
  we finally get
  \begin{align} \label{eq:q_estF}
    p^{\om}_{\th}(t, x, y)
    \;\leq\;
    \frac{C_3}{t^{d/2}}\, \bigg( 1 + \frac{d(x_0,x)}{\sqrt t} \bigg)^{\!\!\ga}
    \bigg( 1 + \frac{d(x_0,y)}{\sqrt t} \bigg)^{\!\!\ga}\,
    \exp\!\bigg( 
      d_{\th}^{\om}(x, y)\, 
      F\bigg(\frac{d_{\th}^{\om}(x, y)}{2C_{\mathrm{deg}}t}\bigg) 
    \bigg).
  \end{align}
  Further, notice that
  \begin{align*}
    F(s)
    \;=\;
    s^{-1} \big( (1 + s^s)^{1/2} - 1 \big)
    \,-\, \log\big( s+(1+s^2)^{1/2}\big)
  \end{align*}
  and $F(s)\leq -s/2(1-s^2/10)$ for $s>0$ (see \cite{BD10} and \cite[page 70]{Da93}).  Hence, if $s\leq 3$, then $F(s)\leq -s/20$ whereas if $s \geq \me$, then
  \begin{align*}
    F(s)
    \;\leq\;
    1 - \log(2s)
    \;=\;
    -\log(2s/\me).
  \end{align*}
  Now, choose $x = x_0$.  In view of \eqref{eq:q_estF} we find suitable constants $c_1, \ldots, c_3$ such that if $d_{\th}^{\om}(x_0, y) \leq c_1 t$ then
  \begin{align*}
    p^{\om}_{\th}(t, x_0, y)
    \;\leq\;
    c_2\, t^{-d/2}\,
    \bigg( 1 + \frac{d(x_0,y)}{\sqrt t} \bigg)^{\!\!\ga}\,
    \exp\!\big( \!-\! 2c_3\, d_{\th}^{\om}(x_0, y)^2/t\,\big).
  \end{align*}
  This finishes the proof of (i).  In the case $d_{\th}^{\om}(x_0, y) \geq c_1 t$ statement (ii) can be obtained from \eqref{eq:q_estF} by similar arguments.
\end{proof}

\bibliographystyle{abbrv}
\bibliography{literature}

\begin{thebibliography}{10}

\bibitem{ADS15}
S.~Andres, J.-D. Deuschel, and M.~Slowik.
\newblock {Invariance principle for the random conductance model in a
  degenerate ergodic environment}.
\newblock {\em Ann. Probab.}, 43(4):1866--1891, 2015.

\bibitem{ADS16}
S.~Andres, J.-D. Deuschel, and M.~Slowik.
\newblock {Harnack inequalities on weighted graphs and some applications to the
  random conductance model}.
\newblock {\em Probab. Theory Related Fields}, 164(3-4):931--977, 2016.

\bibitem{ADS16a}
S.~Andres, J.-D. Deuschel, and M.~Slowik.
\newblock {Heat kernel estimates for random walks with degenerate weights}.
\newblock {\em Electron. J. Probab.}, 21:Paper No. 33, 21, 2016.

\bibitem{AN17}
S.~Andres and S.~Neukamm.
\newblock {Berry-Esseen Theorem and Quantitative homogenization for the Random
  Conductance Model with degenerate Conductances}.
\newblock {\em Preprint, available at arXiv:1706.09493}, 2017.

\bibitem{BD10}
M.~T. Barlow and J.-D. Deuschel.
\newblock {Invariance principle for the random conductance model with unbounded
  conductances}.
\newblock {\em Ann. Probab.}, 38(1):234--276, 2010.

\bibitem{BBHK08}
N.~Berger, M.~Biskup, C.~E. Hoffman, and G.~Kozma.
\newblock {Anomalous heat-kernel decay for random walk among bounded random
  conductances}.
\newblock {\em Ann. Inst. Henri Poincar{\'e} Probab. Stat.}, 44(2):374--392,
  2008.

\bibitem{BO12}
M.~Biskup and O.~Boukhadra.
\newblock {Subdiffusive heat-kernel decay in four-dimensional i.i.d. random
  conductance models}.
\newblock {\em J. Lond. Math. Soc. (2)}, 86(2):455--481, 2012.

\bibitem{BKM15}
O.~Boukhadra, T.~Kumagai, and P.~Mathieu.
\newblock {Harnack inequalities and local central limit theorem for the
  polynomial lower tail random conductance model}.
\newblock {\em J. Math. Soc. Japan}, 67(4):1413--1448, 2015.

\bibitem{CKS87}
E.~A. Carlen, S.~Kusuoka, and D.~W. Stroock.
\newblock {Upper bounds for symmetric {M}arkov transition functions}.
\newblock {\em Ann. Inst. H. Poincar{\'e} Probab. Statist.}, 23(2,
  suppl.):245--287, 1987.

\bibitem{Da89}
E.~B. Davies.
\newblock {\em {Heat kernels and spectral theory}}, volume~92 of {\em
  {Cambridge Tracts in Mathematics}}.
\newblock Cambridge University Press, Cambridge, 1989.

\bibitem{Da93}
E.~B. Davies.
\newblock {Large deviations for heat kernels on graphs}.
\newblock {\em J. London Math. Soc. (2)}, 47(1):65--72, 1993.

\bibitem{De99}
T.~Delmotte.
\newblock {Parabolic {H}arnack inequality and estimates of {M}arkov chains on
  graphs}.
\newblock {\em Rev. Mat. Iberoamericana}, 15(1):181--232, 1999.

\bibitem{DF19}
J.-D. Deuschel and R.~Fukushima.
\newblock Quenched tail estimate for the random walk in random scenery and in
  random layered conductance {II}.
\newblock {\em in preparation}.

\bibitem{DF16}
J.-D. Deuschel and R.~Fukushima.
\newblock Quenched tail estimate for the random walk in random scenery and in
  random layered conductance.
\newblock {\em Stochastic Process. Appl.}, 129(1):102--128, 2019.

\bibitem{DRS14}
A.~Drewitz, B.~R\'{a}th, and A.~Sapozhnikov.
\newblock On chemical distances and shape theorems in percolation models with
  long-range correlations.
\newblock {\em J. Math. Phys.}, 55(8):083307, 30, 2014.

\bibitem{EKM97}
P.~Embrechts, C.~Kl{\"u}ppelberg, and T.~Mikosch.
\newblock {\em {Modelling extremal events}}, volume~33 of {\em {Applications of
  Mathematics (New York)}}.
\newblock Springer-Verlag, Berlin, 1997.

\bibitem{Fo11}
M.~Folz.
\newblock {Gaussian upper bounds for heat kernels of continuous time simple
  random walks}.
\newblock {\em Electron. J. Probab.}, 16:no. 62, 1693--1722, 2011.

\bibitem{KK77}
S.~N. {Kru\v zkov} and I.~M. {Kolod\=\i \u\i}.
\newblock {A priori estimates and {H}arnack's inequality for generalized
  solutions of degenerate quasilinear parabolic equations}.
\newblock {\em Sibirsk. Mat. \v Z.}, 18(3):608--628, 718, 1977.

\bibitem{Mo09}
J.-C. Mourrat.
\newblock {Variance decay for functionals of the environment viewed by the
  particle}.
\newblock {\em Preprint, available at arXiv:0902.0204v2}, 2009.

\bibitem{S-C02}
L.~Saloff-Coste.
\newblock {\em {Aspects of {S}obolev-type inequalities}}, volume 289 of {\em
  {London Mathematical Society Lecture Note Series}}.
\newblock Cambridge University Press, Cambridge, 2002.

\bibitem{Zh11}
V.~V. Zhikov.
\newblock {Estimates of {N}ash-{A}ronson type for degenerate parabolic
  equations}.
\newblock {\em Sovrem. Mat. Fundam. Napravl.}, 39:66--78, 2011.

\end{thebibliography}

\end{document}